\documentclass{article}
\usepackage[utf8]{inputenc}
\usepackage[margin=1in]{geometry}
\usepackage{amsmath} 
\usepackage{amsthm}
\usepackage{todonotes}
\usepackage[printonlyused]{acronym}
\usepackage{xcolor}
\usepackage{amsmath}
\usepackage{amssymb}
\usepackage{mathtools}
\usepackage{subcaption}
\usepackage{hyperref}
\usepackage{mathrsfs}
\newcommand{\rw}[1]{{#1}}

\hypersetup{colorlinks=true, linkcolor=black}
 
\newcommand{\R}{\mathbb{R}} 
\newcommand{\K}{\mathbb{K}} 
 
\newcommand{\N}{\mathbb{N}}

\newcommand{\dc}{\mathcal{D}}
\newcommand{\Lie}{\mathcal{L}}

\newcommand{\1}{\mathbf{1}}
\newcommand{\psd}{\mathbb{S}}

\DeclarePairedDelimiter{\abs}{\lvert}{\rvert}
\DeclarePairedDelimiter{\norm}{\lVert}{\rVert}

\DeclarePairedDelimiter{\floor}{\lfloor}{\rfloor}

\DeclarePairedDelimiterX{\inp}[2]{\langle}{\rangle}{#1, #2}
\DeclarePairedDelimiter{\supp}{\textrm{supp}(}{)}
\DeclarePairedDelimiter{\Mp}{\mathcal{M}_+(}{)}

\newtheorem{thm}{Theorem}[section]
\newtheorem{lem}[thm]{Lemma}
\newtheorem{prop}[thm]{Proposition}

\newtheorem{cor}{Corollary}
\newtheorem{defn}{Definition}[section]

\newtheorem{rmk}{Remark} 

\title{\LARGE \bf Quantifying the Safety of Trajectories \\using Peak-Minimizing Control
}

\renewcommand\footnotemark{}

\author{Jared Miller$^1$, Mario Sznaier$^1$
\thanks{$^1$J. Miller, and M. Sznaier are with the Robust Systems Lab,  ECE Department, Northeastern University, Boston, MA 02115. (e-mails: miller.jare@northeastern.edu, msznaier@coe.neu.edu).}
\thanks{J. Miller and M. Sznaier were partially supported by NSF grants   ECCS--1808381 and CNS--2038493, AFOSR grant FA9550-19-1-0005, and ONR grant N00014-21-1-2431.  
J. Miller was in part supported by the Chateaubriand Fellowship (performed at LAAS-CNRS) of the Office for Science \& Technology of the Embassy of France in the United States, and by the International Student Exchange Program from AFOSR.}
}

\begin{document}

\maketitle

\begin{abstract}
\label{sec:abstract}
This work quantifies the safety of trajectories of a dynamical system by the perturbation intensity required to render a system unsafe (crash into the unsafe set). 
In the data-driven setting, this perturbation intensity can be interpreted as the minimal data corruption required for a data-consistent model to crash.
Computation of this measure of safety is posed as a peak-minimizing optimal control problem. Convergent lower bounds on the minimal peak value of controller effort are computed using polynomial optimization and the moment-Sum-of-Squares hierarchy. 


\end{abstract}
\section{Introduction}
\label{sec:introduction}

A trajectory starting at an initial point $x_0 \in X$ following dynamics $\dot{x}= f_0(t, x)$ is safe with respect to the unsafe set $X_u$ in the time horizon $t \in [0, T] \subset [0, \infty)$ if there does not exist a time $t'$ such that $x(t' \mid x_0)$ is a member of $X_u$. The set $X_0$ is safe with respect to $X_u$ if all initial points $x_0 \in X_0$ generate safe trajectories. 
This paper quantifies the safety of trajectories by maximum control effort (\ac{OCP} cost) needed to crash the agent into the unsafe set. 
An example of this type of safety result is if tilting a car's steering wheel by a maximum extent of  $3^\circ$ over the course of its motion would cause the car to crash. In the data-driven framework, a trajectory is labeled safe if it would require the true system to have a large constraint violation against any of its state-derivative data observations in order to crash.
The process of analyzing safety by computing the peak-minimizing-\ac{OCP} cost will be referred to as `crash safety'.

Let $W \subset \R^L$ be a compact input set, and let $\mathcal{W}$ be the class of functions whose graphs satisfy $\rw{\forall t \in [0, T]: \ w(t) \in W}$.
Given a control-cost $J(w)$ (which can represent the intensity of the data corruption), we can pose the following peak-minimizing free-terminal-time \ac{OCP}:
\begin{equation}
    \label{eq:crash_traj}
    \begin{aligned}
    Q^* = & \inf_{t, \ x_0, \ w}  \ \sup_{t' \in [0, t]} J(w(t'))\\
    & \dot{x}(t') =  f(t', x(t'), w(t')) & & \forall t' \in [0, T] \\
    & x(t \mid x_0, w(\cdot)) \in X_u \\ 
    & w(\cdot) \in \mathcal{W}, \ t \in [0, T], \ x_0 \in X_0.
    \end{aligned}
\end{equation}


The variables of \eqref{eq:crash_traj} are the stopping time $t$, the initial condition $x_0$, and the input process $w(\cdot)$. 
Assuming for the purposes of this introduction that $J(0)=0, \ \forall w \neq 0: J(w) > 0$ and that $J$ possesses connected \rw{sublevel} sets \rw{(such as $J$ quadratic)},
the set $x_0$ is unsafe if $Q^* = 0$ because the process $w(t) = 0$ is sufficient for the trajectory to reach a terminal set of $X_u$. 
The value of a nonzero $Q^*$ then measures the amount of control effort (perturbation intensity) needed to render the trajectory unsafe. Connected level sets are imposed to add interpretability to $Q^*$; a disconnected choice of $J$ with multiple local minima could yield a large input $w$ with a low $Q^*$.

A running cost $\int_{0}^{T} J(w(t')) dt$ yielding a standard-form (Lagrange) \ac{OCP} may also be applied, but we elect to use a peak-minimizing cost $\max_{t'} J(w(t'))$ in order to penalize perturbation intensity.
The running-cost would penalize a low-magnitude control being applied for an extended period of time, while peak-minimizing control reduces the intensity. 

Peak-minimizing control problems, such as in  \eqref{eq:crash_traj}, are a particular form of robust optimal control in which the minimizing agents are $(t, x_0, w(\cdot))$ and the maximizing agent is $t' \in [0, t]$. Necessary conditions for these robust programs may be found in \cite{vinter2005minimax}. Instances of  peak-minimizing control include minimizing the maximum number of infected persons in an epidemic under budget constraints \cite{molina2022optimal} and choosing flight parameters to minimize the maximum skin temperature during atmospheric reentry \cite{lu1988minimax, kreim1996minimizing}. The work in \cite{molina2022equivalent} outlines conversions between peak-minimizing \acp{OCP} and equivalent Mayer-form \acp{OCP} (terminal cost only). 

This paper continues a sequence of research about quantifying the safety of trajectories. 
Unsafety can be proven using path-planning by finding a feasible pair $(t', x_0) \in [0, T] \times X_0$ such that $x(t' \mid x_0) \in X_0$.
Barrier \cite{prajna2004safety, prajna2006barrier} and Density functions \cite{rantzer2004analysis} are binary certificates confirming that there does not exist an unsafe trajectory based on the satisfaction of nonnegativity constraints. Safety margins use maximin peak estimation to estimate the $X_u$-representing-inequality-constraint violation \cite{miller2020recovery}. The distance of closest approach between a trajectory starting in $X_0$ and points in $X_u$ is a more interpretable measure of safety than abstract safety margins \cite{miller2021distance}. Even so, distance estimation does not tell the full story; a trajectory may lie close to $X_u$ in the sense of distance, but it could require a large value of $Q^*$ to render the same trajectory unsafe.

Direct solution of \acp{OCP} using the \ac{HJB} equation or the Pontryagin Maximum Principle may be challenging, especially when solutions do not exist in closed form \cite{liberzon2011calculus}. These generically non-convex \acp{OCP} may be lifted into convex infinite-dimensional \acp{LP} in occupation measures \cite{lewis1980relaxation}, whose dual \ac{LP} involve subvalue functions satisfying \ac{HJB} inequalities.
These infinite-dimensional \acp{LP} produce lower-bounds on the true \ac{OCP}, with equality holding under compactness and regularity conditions.
The moment-\ac{SOS} hierarchy of \acp{SDP} may be used to produce a rising sequence of lower bounds to the true \ac{OCP} if the dynamics $f(t, x, w)$ are polynomial and the sets $(W, X, X_0, X_u)$ are \ac{BSA} \cite{henrion2008nonlinear}. This infinite-dimensional \ac{LP} and finite-dimensional \ac{SDP} pattern has also been applied to reachable set estimation \cite{henrion2013convex, korda2013inner}, peak estimation \cite{fantuzzi2020bounding, miller2021uncertain}, and maximum controlled invariant set estimation \cite{korda2014convex}.

This paper transforms program \eqref{eq:crash_traj} into the Mayer \ac{OCP} using \cite{molina2022equivalent}, relaxes the nonconvex \ac{OCP} into an infinite-dimensional \ac{LP} with the same objective value \cite{lewis1980relaxation}, and then lower-bounds $Q^*$ by using the moment-\ac{SOS} hierarchy \cite{henrion2008nonlinear, lasserre2009moments}. The robust counterpart method of  \cite{ben2009robust, miller2023robustcounterpart} will be used to simplify the infinite-dimensional \ac{LP} when $W$ and the graph of $J$ are both polytopic.

Contributions of this work include,
\begin{itemize}
    \item An interpretation of peak-minimal control costs as a quantification of safety
    \item An infinite-dimensional \ac{LP} that produces the peak-minimal cost under compactness, regularity, and convexity conditions
    \item A subvalue functional that lower-bounds the crash-safety effort
    \item An application of crash-safety towards data-driven safety analysis
\end{itemize}

This paper has the following structure: 
Section \ref{sec:preliminaries} reviews the notations, the peak-minimizing control framework of \cite{molina2022equivalent}, and \ac{SOS} methods. Section \ref{sec:crash_lp} formulates an infinite-dimensional \ac{LP} to solve \eqref{eq:crash_traj}. Section \ref{sec:crash_robust} applies the crash-safety framework towards $L_\infty$-penalized data-driven analysis using robust Lie counterparts from 
 \cite{miller2023robustcounterpart}. Section \ref{sec:crash_sos} forms \ac{SOS} programs for crash-safety and tabulates their computational complexity. 
Section \ref{sec:crash_subvalue} evaluates the safety of points inside $X$ by a subvalue function of the crash-safety cost. 
 Section \ref{sec:crash_examples} provides demonstrations of crash-safety. 
 Section \ref{sec:conclusion} concludes the paper.
\section{Preliminaries}
\label{sec:preliminaries}

\subsection{Acronyms/Initialisms}
\begin{acronym}[WSOS]
\acro{BSA}{Basic Semialgebraic}


\acro{LMI}{Linear Matrix Inequality}
\acroplural{LMI}[LMIs]{Linear Matrix Inequalities}
\acroindefinite{LMI}{an}{a}

\acro{HJB}{Hamilton-Jacobi-Bellman}

\acro{LP}{Linear Program}
\acroindefinite{LP}{an}{a}

\acro{OCP}{Optimal Control Problem}
\acroindefinite{OCP}{an}{a}





\acro{PD}{Positive Definite}

\acro{PSD}{Positive Semidefinite}

\acro{SDP}{Semidefinite Program}
\acroindefinite{SDP}{an}{a}

\acro{SDR}{Semidefinite Representable}
\acroindefinite{SDR}{an}{a}

\acro{SOS}{Sum of Squares}

\acro{WSOS}{Weighted Sum of Squares}

\end{acronym}

\subsection{Notation}

The set of real numbers is $\R$ and the $n$-dimensional real vector spaces is $\R^n$. The all-ones vector is $\1$. The set of natural numbers is $\N$ and the set of $n$-dimensional multi-indices is $\N^n$. The set of natural numbers between $a$ and $b$ is $a..b \subset \N$. The cone of $n \times n$ symmetric \ac{PSD} matrices is $\psd^n_+$.

The set of polynomials of an indeterminate $x$ with real-valued coefficients is $\R[x]$. The degree of a polynomial $p \in \R[x]$ is $\deg p$. The vector space of polynomials up to degree $d \in \N$ is $\R[x]_{\leq d}$. The coefficients of a polynomial $p \in \R[x]$ are $\text{coeff}_x(p(x))$.

The ring of continuous functions over a space $S \subseteq \R^n$ is $C(S)$. The set of first-differentiable functions over $S$ is $C^1(S) \subset C(S)$. The subcone of nonnegative functions over $S$ is $C_+(S) \subset C(S)$.

The set of nonnegative Borel measures over $S$ is $\Mp{S}$. Given a measure $\mu \in \Mp{S}$, the support $\supp{\mu}$ is the locus of points $s' \in S$ such that every open neighborhood of $s'$ has a nonzero measure with respect to $\mu$.
A pairing $\inp{\cdot}{\cdot}$ may be defined between $f \in C(S)$ and $\mu \in \Mp{S}$ by $\inp{f}{\mu} = \int_{S} f(s) d \mu(s)$. This pairing defines an inner product between $C_+(S)$ and $\Mp{S}$ when the set $S$ is compact. The mass of a measure $\mu \in \Mp{S}$ is $\inp{1}{\mu},$ and $\mu$ is a probability measure if $\inp{1}{\mu}=1$. The Dirac delta $\delta_{s'}$ is the unique probability supported only at $s' \in S$, with $\forall f \in C(S): \inp{f}{\delta_{s'}} = s'$. Given a curve $s: [0, T] \times S$, the occupation measure $\mu_{s}$ of $s(t)$ in the times $[0, T]$ is the unique measure satisfying $\forall v \in C([0, T] \times S): \inp{v(t, s)}{\mu_s} = \int_{0}^T v(t, s(t)) dt$.

\subsection{Sum of Squares}

Verifying that a polynomial $p \in \R[x]$ is nonnegative $\forall x \in \R^n$ is generically an NP-hard problem (except for $p$ quadratic, univariate, or bivariate quartic) \cite{hilbert1888darstellung}. A sufficient condition for $p$ to be nonnegative is if there exists $N$ factors $\{p_k \in \R[x]\}_{k=1}^N$ such that $p = \sum_{k=1}^N p_k(x)^2$. Such a $p$ is therefore called an \ac{SOS} polynomial. The cone of \ac{SOS} polynomials is $\Sigma[x]$, and the subset of degree $\leq 2d$ \ac{SOS} polynomials is $\Sigma[x]_d \subset \Sigma[x]$. To each polynomial $p \in \Sigma[x]_d$, there exists an $s$-dimensional vector of polynomials $v(x)$ (e.g. monomials up to degree $d$) and \iac{PSD} \textit{Gram} matrix $Q \in \psd_+^s$ such that $p(x) = v(x)^T Q v(x)$. When the monomial basis is used, the Gram matrix has dimension $\binom{n+d}{d}$. Verification of $p \in \Sigma[x]_{2d}$ at fixed degree can be performed by solving \iac{SDP}. The per-iteration scaling of Interior Point Methods for solving this \ac{SDP}  rises in a jointly polynomial manner with $n$ and $d$ (with  $O(n^{6d})$ and $O(d^{4n})$) \cite{alizadeh1995interior} \cite{miller2022eiv_short}.

\Iac{BSA} set $\K \in \R^n$ is a set formed by the locus of a finite number of inequality and equality constraints of bounded degree:
\begin{align}
\label{eq:bsa_set}
    \K &= \{x \mid \forall i=1..N_g: \ g_i(x)\geq 0, \ \forall i=1..N_h: \ h_j(x)=0\}.
\end{align}
A sufficient condition for $p \in \R[x]$ to be nonnegative over $\K$ is that there exists multipliers $\sigma_0, \sigma_i, \phi_j$ such that
\cite{putinar1993compact}
\begin{subequations}
\label{eq:putinar}
    \begin{align}
        & p(x) = \sigma_0(x) + \textstyle \sum_i {\sigma_i(x)g_i(x)} + \textstyle \sum_j {\theta_j(x) h_j(x)}\\
        &\exists  \sigma_0(x) \in \Sigma[x], \quad \sigma_i(x) \in \Sigma[x], \quad \phi_j \in \R[x]. \label{eq:putinar_variables}
    \end{align}
\end{subequations}
The \ac{WSOS} cone $\Sigma[\K]$ is the set of polynomials that admit a representation as in \eqref{eq:putinar}. The truncated \ac{WSOS} set $\Sigma[\K]_{d}$ is the set of polynomials where the certificate in \eqref{eq:putinar} has $\deg \sigma_0 \leq 2d, \ \forall i: \deg \sigma_i g_i \leq 2d, $ and $\forall j: \deg \phi_j h_j \leq 2d$. Verification of $p \in \Sigma[\K]$ could require multipliers that have an exponential degree in $n$ and $\deg p$ \cite{nie2007complexity}.

The \ac{BSA} set $\K$ is \textit{Archimedean} if there exists $R \in [0, \infty)$ such that $R - \norm{x}_2^2 \in \Sigma[\K]$. Archimedeanness is a stronger property than compactness \cite{cimpric2011closures}. If there exists a known $\tilde{R}$ verifying compactness such that $\K \subseteq \{x \in \R^n \mid \norm{x}_2^2 \leq \tilde{R}\},$ then the compact set $\K$ may be rendered Archimedean by adding the redundant ball constraint $\tilde{R} - \norm{x}^2_2 \geq 0$ to the list of constraints in \eqref{eq:bsa_set}. When $\K$ is Archimedean, the Putinar Positivestellensatz states that every positive polynomial over $\K$ is a member of $\Sigma[\K]$ \cite[Theorem 1.3]{putinar1993compact}. The moment-\ac{SOS} hierarchy is the process of increasing the degree $d$ in $\Sigma[\K]_{d}$ to eventually include the set of all positive polynomials.


\section{Crash-Safety Program}
\label{sec:crash_lp}

This section applies peak-minimizing control conversion to the general crash-safety task in \eqref{eq:crash_traj}. The specific crash-safety task used in the data-driven framework will be subsequently considered in Section \ref{sec:crash_robust}.

\subsection{Motivating Example}

This subsection provides an example demonstrating how \eqref{eq:crash_traj} can be used for safety quantification.
The Flow dynamics from \cite{rantzer2004analysis} are
\begin{align}
    \label{eq:flow}
    \dot{x} = \begin{bmatrix}x_2 \\ -x_1 -x_2 + \frac{1}{3}x^3_1\end{bmatrix}.
\end{align}
This example will perturbed \eqref{eq:flow} by an uncertainty process restricted to $\forall t: \ w(t) \in [-1, 1]$:
\begin{align}
    \label{eq:flow_w1}
    \dot{x} = \begin{bmatrix}x_2 \\ -x_1 -x_2 + \frac{1}{3}x^3_1\end{bmatrix} + w \begin{bmatrix}
        0 \\ 1
    \end{bmatrix}.
\end{align}

Trajectories evolve over a time horizon of $T=5$ in the state set $X = [-0.6, 1.75] \times [-1.5,1.5]$ with a maximum corruption of $J_{\max}=2$. 
System dynamics are illustrated by the blue streamlines in Figure \ref{fig:same_dist}. The red half-circle is the unsafe set $X_u = \{x \mid x_2 \leq -0.5, \ (x_1 - 1)^2 + (x_2+0.5)^2 \leq 0.5^2\}$. Two trajectories of this system are highlighted. The green trajectory starts from the top initial point $X^1_0 = [0; 1],$ and the yellow trajectory starts from the bottom initial point $X^2_0 = [1.2966 -1.5]$. The distance of closest approach to $X_u$ is $0.2498$ for both trajectories (matching up to four decimal places). The $0.2498$-contour of constant distance is displayed by the red curve surrounding $X_u$.

\begin{figure}[h]
    \centering
    \includegraphics[width=0.6\linewidth]{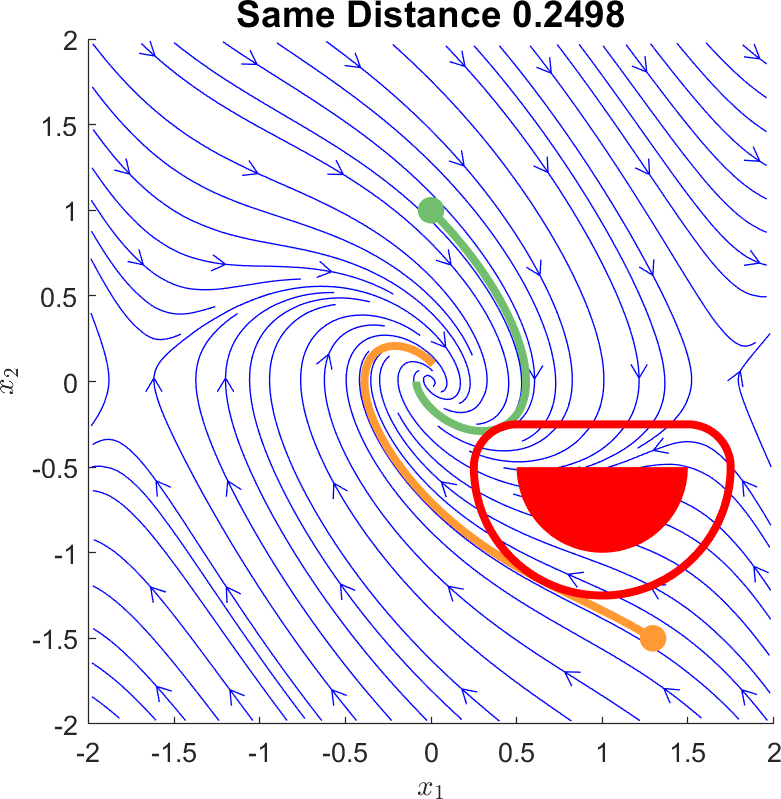}
    \caption{Two trajectories with nearly the same distance but different crash-bounds}
    \label{fig:same_dist}
\end{figure}

The \ac{OCP} solver CasADi \cite{andersson2019casadi} returns approximate bounds for \eqref{eq:crash_traj} of $Q^* \approx 0.3160$ for $X_0^1$ (green) and $Q^* \approx 0.6223$ for $X_0^2$ (yellow).  The points $(X^1_0, X^2_0)$ return nearly identical distances of closest approach, but $X^2_0$ may be judged as safer than $X^1_0$ under the disturbance model in \eqref{eq:flow} due to its higher crash-bound value. Degree-4 \ac{SOS} tightenings of \eqref{eq:crash_traj_z} developed in the sequel return lower bounds of $0.3018$ and $0.5273$ respectively.

\subsection{Assumptions}
We will require \rw{assumptions in order to ensure that the developed infinite-dimensional convex \ac{LP} will have the same optimal value as \eqref{eq:crash_traj}. These assumptions compactnes, convexity, and regularity assumptions arise from the optimal control work in \cite{vinter1978equivalence}.}

\begin{enumerate}
    \item[A1] The sets $[0,T], [0, J_{\max}], X, W, X_u, X_0$ are all compact.
    \item[A2] The image $f(t, x, W)$ is convex for each fixed $(t, x)$.
    \item[A3] The dynamics function $f(t, x, w)$ is Lipschitz in the compact domain $[0, T] \times X \times W$.
    \item[A4] If $x(t \mid x_0, w) \, \in \partial X$ for some $t \in [0, T], \ x_0 \in X_0, \ w \in \mathcal{W}$, then $x(t' \mid x_0) \not\in X \ \forall t' \in (t, T]$.
\end{enumerate}

\rw{A1 ensures boundedness and boundary-inclusion of all relevant sets. A2 guarantees that it suffices to analyze trajectories of the differential inclusion $\dot{x} \in \text{conv}(f(t, x, W))$. This paper will generally deal with the case where $f$ is affine in $w$, so that A2 will therefore be satisfied. A3 ensures Lipschitz regularity (and thus uniqueness) of trajectories given an input $w(\cdot)$.}
A4 is an assumption of non-return used in \cite{miller2021distance} that is weaker than ensuring $X$ is an invariant set.


\subsection{Formulation}

We use the peak-minimizing control conversion of \cite{molina2022equivalent} on program \eqref{eq:crash_traj}.

\begin{thm}
\label{thm:same_crash_z}
The following program has the same optimal value as \eqref{eq:crash_traj}:
\begin{subequations}
    \label{eq:crash_traj_z}
    \begin{align}
    Q^*_z = & \inf_{t, \ x_0, \ z, \ w}  \ z    \label{eq:crash_traj_z_obj}\\
    & \dot{x}(t') =  f(t', x(t'), w(t')) & 
    & \forall t' \in [0, T]\\
    & \dot{z}(t') = 0 \label{eq:crash_traj_z_const}& 
    & \forall t' \in [0, T] \\
    & J(w(t')) \leq z & &\forall t' \in [0, T]  \label{eq:crash_traj_z_supp}\\
    & x(t \mid x_0, w(\cdot)) \in X_u \\ 
    & w(\cdot) \in W, \ t \in [0, T] \\
    & x_0 \in X_0,  z \in [0, J_{\max}].
    \end{align}
\end{subequations}
\end{thm}
\begin{proof}

\rw{The parameter $z$ is added as a constant state \eqref{eq:crash_traj_z_const}. This parameter upper-bounds the values of $J(w)$ along trajectories \eqref{eq:crash_traj_z_supp}. The objective \eqref{eq:crash_traj_z_obj} therefore infimizes this upper-bound of $J(w)$, forming an equivalence with  \eqref{eq:crash_traj}.}

\end{proof}

\subsection{Crash-Safety Linear Program}

\rw{This subsection will present infinite-dimensional \ac{LP} relaxation to \eqref{eq:crash_traj_z}, and will prove conditions for which the \ac{LP} relaxations give the same objective value as the original problems in trajectories. First,} the following compact support  sets involving the input $w$ and peak-bound $z$ \rw{must be defined}:
\begin{align}
\label{eq:crash_support}
    Z &= [0, J_{\max}] & \Omega &= \{(w, z) \in W \times Z: J(w) \leq z\}.
\end{align}

\rw{The \ac{LP} will be posed in terms of an auxiliary function $v(t,x, z) \in C^1$. \rw{Consistency of values of $v$ along trajectories of the dynamical system will be enforced through a Lie derivative constraint. This}  Lie derivative $\Lie_f$ } will be defined as
\begin{align}
\label{eq:lie}
\Lie_f v(t,x,z,w) = (\partial_t + f(t, x, w)\cdot \nabla_x) v(t, x, z)
\end{align}

\rw{Our presented} \ac{LP} formulation of the crash-safety \ac{OCP} in \eqref{eq:crash_traj}:
\begin{subequations}
\label{eq:crash_cont}
\begin{align}
    q^* = & \sup_{\gamma \in \R, \ v} \ \gamma \label{eq:crash_cont_obj}& & \\
    & \forall (x, z) \in X_0 \times Z: \nonumber \\ 
    & \qquad v(0, x, z) \geq \gamma \ & & \label{eq:crash_cont_0} \\
    & \forall (t, x, z) \in  [0, T] \times X_u \times Z \nonumber \\
    & \qquad v(t, x, z)  \leq z \ & &  \label{eq:crash_cont_p} \\
    & \forall (t, x, z, w)\in [0, T] \times X \times \Omega \nonumber \\
    & \qquad \Lie_{f} v(t, x, z, w) \geq 0 & & \label{eq:crash_cont_lie}\\
    & v(t,x,z) \in C^1([0, T]\times X \times Z). \label{eq:crash_cont_v}& &
\end{align}
\end{subequations} 

\begin{thm}
\label{thm:same_crash_meas}
    Under assumptions A1-A\rw{4}, programs \eqref{eq:crash_traj} and \eqref{eq:crash_cont} will have equal objectives $q^* = Q^*$.
\end{thm}
\begin{proof}
    Program \eqref{eq:crash_traj_z} with optimum $Q^*_z$ is a standard-form \ac{OCP} with free terminal time and zero running cost. Under assumptions A1-A5, Theorem 2.1 of \cite{lewis1980relaxation} proves that $Q^*_z = q^*$. Section 6.3 of \cite{lewis1980relaxation} specifically discusses state-dependent controls (e.g. $(w, z) \in \Omega$). Theorem \ref{thm:same_crash_z} provides that $Q^* = Q^*_z$, which together implies that $Q^* = q^*$.
\end{proof}

The dual \ac{LP} of \eqref{eq:crash_cont} may be phrased in terms of an initial measure $\mu_0$, terminal measure $\mu_u$, and relaxed occupation measure $\mu$:
\begin{subequations}
\label{eq:crash_meas}
\begin{align}
m^* = & \  \inf_{\mu_0, \mu_p, \mu} \quad 
\inp{z}{\mu_u}
 \label{eq:crash_meas_obj} \\
    & \inp{v(t, x, z)}{\mu_u} = \inp{v(0, x, z)}{\mu_0} + \inp{\Lie_f v(t, x, z, w)}{\mu} \label{eq:crash_meas_flow}\\
    & \inp{1}{\mu_0} =  1 \label{eq:crash_meas_prob}\\
    & \mu_0 \in \Mp{X \times Z} \label{eq:crash_meas_init}\\
    & \mu_u \in \Mp{[0, T] \times X_u \times Z} \label{eq:crash_meas_term}\\
    & \mu \in \Mp{[0, T] \times X \times  \Omega }. \label{eq:crash_meas_occ}
\end{align}
\end{subequations}
Constraint \eqref{eq:crash_meas_obj} is a Liouville equation involving the Young measure $\mu$ \cite{young1942generalized}.

\begin{lem}
\label{lem:feas_meas}
There exists a feasible solution to \eqref{eq:crash_meas_flow}-\eqref{eq:crash_meas_occ} under A1-A4.
\end{lem}
\begin{proof}
    Let $t^* \in [0, T]$ be a stopping time, $x_0 \in X_0$ be an initial condition, and $w(\cdot) \in \mathcal{W}$ be an input such that $x(t^* \mid x_0, w(\cdot)) \in X_u$. Let $z^*$ be a feasible solution to $\forall t \in [0, t^*]: (z^*, w(t)) \in \Omega$. Then the probability measures can be set to  $\mu_0 = \delta_{x=x_0, z=z^*}$ and $\mu_u = \delta_{t = t^*, x=x(t^* \mid x_0, w(\cdot)), \ z=z^*}$, and $\mu$ can be assigned to the occupation measure of $t \mapsto (t, x(t^* \mid x_0, w(\cdot)), w(t))$ in the times $[0, t^*]$.
\end{proof}

\begin{rmk}
The process of \ref{lem:feas_meas} to generate a feasible measure solution may be used when only A1 and A4 are active, thus certifying that $m^* \leq Q^*$.
\end{rmk}

\begin{thm}
\label{thm:crash_duality}
    Strong duality occurs with $q^*=d^*$ between \eqref{eq:crash_meas} and \eqref{eq:crash_cont} under assumptions A1-A4.
\end{thm}
\begin{proof}
    This holds by standard \ac{OCP} \ac{LP} duality arguments from \cite{ lewis1980relaxation, vinter1985dynamic, lasserre2008nonlinear}, in which feasibility of a measure solution is demonstrated in Lemma \ref{lem:feas_meas}.
\end{proof}
\section{Data-Driven Crash-Safety Analysis}
\label{sec:crash_robust}

This section motivates crash-safety in the context of data-driven analysis. \rw{The cost function $J$ will be interpreted as a measure of alignment with a model class, and the assumption that  $J(0)=0$ will be removed in this section to allow for modeling errors}.

\subsection{Data-Driven Overview}
In this section, we will assume that $N_s$ time-state-derivative data  records $\mathcal{D} = \{(t_k, x_k, y_k)\}_{k=1}^{N_s}$ are provided for the true system $\dot{x} = F(t, x)$. \rw{The record $y_k$ is a noisy observation of the derivative (value of $F$) at time $t_k$ and state $x_k$}. The data records in $\dc$ are corrupted by $L_\infty$-bounded \rw{uncertainty} of intensity $\epsilon$ with
\begin{align}  \label{eq:linf_corrupt}
    \forall k = 1..N_s & & \norm{y_k - F(t_k, x_k)}_\infty \leq \epsilon.
\end{align}

We are given a dictionary of functions $(f_0, \{f_\ell\}_{\ell=1}^L)$ that are Lipschitz in $[0, T] \times X$ (e.g. monomials). \rw{This dictionary of functions will be used to describe the ground-truth (and a-priori unknown) dynamics $F(t, x)$. }
We are also given the knowledge that there exists at least one ground-truth choice of parameters $w^* \in \R^L$ with 
\begin{align}
   \textstyle  F(t, x) = f_0(t, x) + \sum_{\ell=1}^L w^*_\ell f_\ell(t, x). \label{eq:dynamics_affine_true}
\end{align}
\rw{The affine parameters $w$ in the data-driven case will be treated as an uncertainty $w(\cdot)$ from the prior section. Specifically, the data $\mathcal{D}$ will be used to describe a set of parameters $w$ that are consistent with the data up to corruption level $z$.}


In the $L_\infty$-bounded polytopic framework, the crash-safety problem \eqref{eq:crash_traj_z} finds an infimal upper bound on the data corruption needed to crash into the unsafe set:
\begin{subequations}
    \label{eq:crash_base}
    \begin{align}
    Z^* = & \inf_{t, \, x_0, \, z, \, w} z\\
    & \forall t' \in [0, T]: \nonumber \\
    & \qquad \dot{x}(t') =  f_0(t', x) + \textstyle \sum_{\ell=1}^L w_\ell f_\ell(t', x(t')) \\
    & \qquad \dot{z}(t')=0  \\
    & x_0 \in X_0, \ x(t \mid x_0, w) \in X_u \\ 
    &  \forall k = 1..N_s:  \\
    & \qquad z \geq \norm{f_0(t_k, x_k) + \textstyle \sum_{\ell=1}^L w_\ell f_\ell(t_k, x_k) - y_k}_\infty \nonumber& &  \label{eq:crash_base_poly}\\
    & z \in Z, \ w \in \R^L, \ t \in [0, T].
    \end{align}
\end{subequations}
If the returned value of \eqref{eq:crash_base} is $Z^* = 0$, then there exists some choice of model parameters $w$ that exactly fit the data $\mathcal{D}$ by \eqref{eq:dynamics_affine_true}. Additionally, this choice $w$ renders at least one trajectory $x(\cdot)$ starting from $X_0$ is unsafe (crashes into $X_0$). Values of $Z^*$ greater than 0 are a certificate of safety in the model structure. A larger value of $Z^*$ indicates that the data must be increasingly corrupted in order to render any trajectory unsafe. Safety is certified if $Z^* > \epsilon$, though we note that the true value of $\epsilon$ may be a-priori unknown.

\subsection{Robust Data-Driven Program}

We will use the input-affine structure of dynamics and polytopic form of  \eqref{eq:crash_base_poly} to form an \ac{LP} that eliminates the \rw{parameter} $w$. This elimination leads to increasingly tractable \ac{SOS} \acp{SDP}.
For each $k=1..N_s$, define the data-record matrices $\Gamma_k, \ h_k$ by
\begin{subequations}
\begin{align}
    \Gamma_k &= \begin{bmatrix}f_1(t_k, x_k), \cdots, f_L(t_k, x_k) \end{bmatrix}  \label{eq:gamma_matrix}\\
    h_k &= f_0(t_k, x_k) - y_k.
\end{align}
\end{subequations}
Letting $\Gamma$ and $h$ be the vertical concatenations of $\{\Gamma_k\}$ and $\{h_k\}$ respectively, we can define the $L_\infty$ performance function and support set as 
\begin{align}
J(w) &= \norm{\Gamma w - h}_\infty = \theta(t, x, w), \quad Z = [0, J_{\max}], \\
\intertext{and the support set for $(w, z)$ from \eqref{eq:crash_base_poly} as}
\label{eq:support_set_crash_data}
    \Omega &= \left\{(w, z) \in \R^L \times Z: \begin{matrix}\Gamma w \leq z\1 - h \\ -\Gamma w \leq z \1 + h\end{matrix}\right\}.
\end{align}

We will eliminate the $w$ variable from \eqref{eq:crash_cont_lie} by introducing new nonnegative multiplier functions $\{\zeta^+, \zeta^-_j\}_{j=1}^{2nT}$. This elimination proceeds using the infinite-dimensional robust counterpart method of \cite{miller2023robustcounterpart}, which requires that \eqref{eq:crash_cont_lie} hold strictly (with a $>0$) constraint.
\begin{thm}
\label{thm:crash_robust_lie}
A strict version of Lie constraint in \eqref{eq:crash_cont_lie} may be robustified  (will have the same feasibility/infeasibility conditions) into
\begin{subequations}
\label{eq:crash_robust_lie}
\begin{align}
& \forall (t, x, z) \in [0, T] \times X \times Z: \nonumber\\
    & \qquad \Lie_{f_0} v - (z\1-h)^T \zeta^+ - (z\1 + h)^T \zeta^- > 0 \label{eq:sub_lie_decomp}\\
    & \forall \ell=1..L: \quad \  (\Gamma^T)_\ell (\zeta^+ - \zeta^-) + f_\ell \cdot \nabla_x v = 0 & &   \\
    & \forall j=1..2nT: \ \zeta^+_j, \zeta^-_j \in C_+([0, T] \times X \times Z).
    \end{align}
    \end{subequations}
\end{thm}
\begin{proof}
We define the following variables 
\begin{subequations}
\label{eq:crash_robust_assign}
\begin{align}
    b_0\rw{(t,x,z)} &= \Lie_{f_0} v\rw{(t,x,z)} & b_\ell\rw{(t,x,z)} &= f_\ell \cdot \nabla_x v\rw{(t,x,z)} & \forall \ell=1..L \\
    A&= [-\Gamma; \Gamma] & e\rw{(z)} &= [z-h; z+h]\\
     K &= \textstyle \prod_{s=1}^{2 n T}\R_{\geq 0}.
\end{align}
\end{subequations}
to express the strict version of \eqref{eq:crash_cont_lie} into the form
\begin{align}
         \textstyle b_0\rw{(t, x, z)} + \sum_{\ell=1}^L w_\ell b_\ell\rw{(t, x, z)} & > 0 & \forall A w + e\rw{(z)} \in K. \label{eq:lin_robust_strict}
\end{align}
The parameters of \eqref{eq:lin_robust_strict} are $(t, x, z) \in [0, T] \times X \times Z$.
By Theorems 4.2 and 4.3 of \cite{miller2023robustcounterpart}, sufficient conditions for \eqref{eq:crash_robust_lie} to equal the strict version of \eqref{eq:crash_cont_lie} are that:
\begin{itemize}
   \item[R1] $K$ is a convex pointed cone.
    \item[R2] $[0, T] \times X \times Z$ is compact.
    \item[R3] $A$ is constant in $(t, x, z)$.
    \item[R4] $(e, b_0, \{b_\ell\})$ are continuous in   $(t, x, z)$.
\end{itemize}

R1 holds because $\R_{\geq0}^{2nT}$ is a convex and pointed cone. Compactness of $[0, T] \times X$ holds by A1, and compactness of $Z$ holds by A2. R3 is true because $\Gamma$ is a constant matrix computed from the data in $\dc$
from \eqref{eq:gamma_matrix}. R4 is satisfied because $e$ is continuous (affine) in $z$, and $( b_0, \{b_\ell\})$ are continuous given that $v \in C^1$ \eqref{eq:crash_cont_v} and $(f_0, \{f_\ell\})$ are Lipschitz (A3). The theorem is proven because R1-R4 are all fullfilled.
\end{proof}

\begin{rmk}
Strictness in \eqref{eq:sub_lie_decomp} is required to ensure that $\zeta^\pm$ may be chosen to be continuous while not adding conservatism. A nonstrict inequality for \eqref{eq:sub_lie_decomp} may be developed  using possibly discontinuous multipliers $\zeta^\pm$.
\end{rmk}

\begin{rmk}
This paper discussed $L_\infty$-bounded uncertainty, resulting in polytopic decomposition of the Lie constraint by Theorem \ref{thm:crash_robust_lie}. Theorems 4.2 and 4.4 of \cite{miller2023robustcounterpart} may be applied when $\Omega$ is a more general semidefinite representable set parameterized by $z$, such as an intersection of ellipsoids for $L_2$-bounded \rw{uncertainty}, or a projection of spectahedra for semidefinite bounded \rw{uncertainty}. 
\end{rmk}

\section{SOS Programs }
\label{sec:crash_sos}

This section poses finite-dimensional \ac{SOS} tightenings to the infinite-dimensional crash-safety programs. 

We will require a strengthening of assumptions A1 \rw{ and A3 in order to use \ac{SOS} methods for crash-safety}:
\begin{itemize}
    \item[A5] The sets $(X, X_u, X_0, [0, T], Z, \Omega)$ are all Archimedean \ac{BSA} sets and the dynamics $f(t, x, w)$ are polynomial.
\end{itemize}

\subsection{Standard Crash-Safety}

For a given degree $d$, define $\tilde{d} = d + \floor{\deg f/2}$ as the dynamics degree of $f(t, x, w)$. The degree-$d$ \ac{SOS} tightening of program \eqref{eq:crash_cont} is
\begin{subequations}
\label{eq:crash_sos}
\begin{align}
    q^*_d = & \max_{\gamma \in \R, \ v} \ \gamma \label{eq:crash_sos_obj}& & \\
    & v(0, x, z) - \gamma  \in \Sigma[X_0 \times Z]_{\leq d}\label{eq:crash_sos_0} \\
    & z - v(t, x, z) \in \Sigma[[0, T] \times X_u \times Z]_{\leq d}\label{eq:crash_sos_p} \\
    & \Lie_{f} v(t, x, z, w) \in \Sigma_{\tilde{d}}[[0, T] \times X \times \Omega]\label{eq:crash_sos_lie}\\
    & v(t,x,z) \in \R [t, x, z]_{\leq 2d}.\label{eq:crash_sos_v}& &
\end{align}
\end{subequations} 

We need to prove boundeness of \eqref{eq:crash_meas} in order to prove convergence of \eqref{eq:crash_sos}.

\begin{lem}
\label{lem:crash_meas_bounded}
All measures $(\mu_0, \mu_u, \mu)$ in \eqref{eq:crash_meas} are bounded under A1-A5.
\end{lem}
\begin{proof}
We will use the sufficient condition that a measure is bounded if it has finite mass and its support set is compact.
All support sets are compact by assumption A1.
    The measure $\mu_0$ has mass 1 by \eqref{eq:crash_meas_prob}. Substitution of $v(t, x, z) = 1$ into \eqref{eq:crash_meas_flow} results in $\inp{1}{\mu_u} = \inp{1}{\mu_0} = 1$, and applying $v(t,x, z)=t$ yields $\inp{1}{\mu} = \inp{t}{\mu_u} \leq T$.
\end{proof}

\begin{thm}
Under assumptions A1-A5, then the sequence of bounds $q^*_d$ will converge as $\lim_{d\rightarrow \infty} q^*_d = Q^*$ to the optimum of \eqref{eq:crash_traj}. 
\end{thm}
\begin{proof}
This convergence will occur by Corollary 8 of \cite{tacchi2022convergence}, along with convergence in Theorem \ref{thm:same_crash_meas}, boundedness of measures in \ref{lem:crash_meas_bounded}, the infinite-dimensional strong duality Theorem \ref{thm:crash_duality}, and strong duality between their finite-dimensional \ac{SDP} truncations  \cite[Arguments from Theorem 4]{henrion2013convex}.
\end{proof}

\subsection{Robust Crash-Safety}

We now apply the robust counterpart from \eqref{eq:crash_robust_lie} to \eqref{eq:crash_cont} in order to form \iac{SOS} program for the $L_\infty$ data-driven scenario. Define $\tilde{d} = d + \max_{\ell\in0..L}{\floor{\deg f_\ell/2}}$ as the dynamics degree of \eqref{eq:dynamics_affine_true}. The $L_\infty$-bounded data-driven robust crash-safety \ac{SOS} tightening at degree $d$ is
\begin{subequations}
\label{eq:crash_sos_robust}
\begin{align}
    \tilde{q}^*_d = & \max_{\gamma \in \R, \ v} \ \gamma \label{eq:crash_sos_robust_obj}& & \\
    & v(0, x, z) - \gamma  \in \Sigma_{d}[X_0 \times Z]\label{eq:crash_sos_robust_0} \\
    & z - v(t, x, z) \in \Sigma_d[[0, T] \times X_u \times Z] \label{eq:crash_sos_robust_p} \\
    & \Lie_{f_0} v - (z \1 - h)^T\zeta^+ - (z \1 + h)^T \zeta^-  \nonumber \\
    & \qquad \qquad \in \Sigma_{\tilde{d}}[[0, T] \times X \times Z]\label{eq:crash_sos_robust_lie}\\
    & \forall \ell=1..\ell:\\
    & \qquad \textrm{coeff}_{txz}((\Gamma^+)_\ell(\zeta^+-\zeta^-) + f_\ell \cdot \nabla_x v)=0 & &  \nonumber\\
    & v(t,x,z) \in \R [t, x, z]_{\leq 2d}.\label{eq:crash_sos_robust_v}& & \\
    & \forall j=1..2nT: \label{eq:crash_sos_robust_zeta}\ \\
    & \qquad \zeta^+_j, \zeta^-_j \in \Sigma[[0, T] \times X \times Z]_{\leq \tilde{d}-1}. & & \nonumber
\end{align}
\end{subequations}

\begin{thm}
    Under assumptions A1-A5 and assuming $L_\infty$ \rw{uncertainty} structure, the sequence of optimal values from \eqref{eq:crash_sos_robust} will converge as $\lim_{d\rightarrow\infty} \tilde{q}_d^* =Q^*$.
\end{thm}
\begin{proof}
    The Lie constraint may be robustified by Theorem \ref{thm:crash_robust_lie}. The \ac{SOS} program in \eqref{eq:crash_sos_robust} will converge to a strict version of \eqref{eq:crash_cont} by Theorem 4.4 of \cite{miller2023robustcounterpart} under the polynomial $v$ restriction. Strictness is not overly restrictive when performing smooth approximations, as shown in the proof of Proposition 5 in \cite{jones2021polynomial}.
\end{proof}

\begin{rmk}
The degree of $\zeta^\pm$ in \eqref{eq:crash_sos_robust_zeta} is set  to $2(\tilde{d}-1)$ so as to ensure that $\deg z\1^T \zeta^\pm = 2\tilde{d}-1 \leq 2\tilde{d}$ in \eqref{eq:crash_sos_robust_lie}.
\end{rmk}

\subsection{Computational Complexity}

\rw{The computational complexity of the \ac{SOS} Programs \eqref{eq:crash_sos} and \eqref{eq:crash_sos_robust} will be compared in terms of the size of their largest \ac{PSD} Gram matrix. Specifically, an SOS constraint $p(x) \in \Sigma_{d}[X]$ with $x \in \R^n$ and $\deg{p} \leq 2d$ has a Gram matrix description of size $\binom{n+d}{d}$, and the per-iteration complexity of interior point methods moment-\ac{SOS} scales as $O(n^{6d})$ or $O(d^{4n})$ \cite{miller2022eiv_short}.}

Program \eqref{eq:crash_sos} has three \ac{WSOS} terms, together leading to Gram matrices of maximal size $\binom{n+1+d}{d},$ \ $\binom{n+2+d}{d}$, and  $\binom{n+L+2+\tilde{d}}{\tilde{d}}$ \cite{lasserre2009moments, miller2022eiv_short}. The performance of \acp{SDP} derived from \eqref{eq:crash_sos} is dominated by the largest size $\ \binom{n+L+2+\tilde{d}}{\tilde{d}}$ and scales as $(n+L+2)^{6 \tilde{d}}$ or $\tilde{d}^{4(n+L+2)}$.

The robustified program in \eqref{eq:crash_sos_robust} breaks up the Lie constraint's maximal-size Gram matrix dimension $\binom{n+L+2+\tilde{d}}{\tilde{d}}$ into one matrix of size $\binom{n+2+\tilde{d}}{\tilde{d}}$ \eqref{eq:crash_sos_robust_lie} and $2nT$ Gram matrices of size $\binom{n+1+\tilde{d}}{\tilde{d}}$ \eqref{eq:crash_sos_robust_zeta}.

\begin{rmk}
The nonredundant face identification method of \cite{caron1989degenerate} requires caution when attempting to reduce complexity of \eqref{eq:crash_sos_robust}. Faces of $W$ that are active at $z_1$ may no longer be active at $z_2 \geq z_1$ or vice versa \cite{loechner1997parameterized}. A bound on \eqref{eq:crash_sos_robust} computed using a subset of faces (constraints) in $\mathcal{D}$ will necessarily be lower than using all faces. This conservatism can be reduced while still eliminating faces by taking the union of active faces of the polytopes in $w$ from $Aw + e \in K$ in \eqref{eq:lin_robust_strict} at a set of values $z \in (0, J_{\max}]$.
\end{rmk}
\section{Subvalue Map}
\label{sec:crash_subvalue}
Program \eqref{eq:crash_cont} returns the worst-case crash safety over a set of initial conditions $X_0$.
We briefly discuss an extension of the crash-safety technique to assessing the safety of arbitrary initial conditions. In the data-driven framework, this could be interpreted as lower-bounding the minimum data corruption needed for a $\rw{\mathcal{D}}$-consistent system to crash when starting at any initial point.

\subsection{Value Functions}

We define the fixed-$z$ value function of \eqref{eq:crash_traj_z} (when starting at $X_0 = x'$) as
\begin{align}
\label{eq:value_v}
    V(x', z) = \begin{cases} z & z \in [0, J_{\max}], \ \exists t\in [0, T], w(\cdot) \in \mathcal{W}:  \\
    & \qquad x(t \mid x_0, w(\cdot)) \in X_u, \ J(w(t')) \leq z\  \forall t'\in [0,t] \\
    \infty & \textrm{otherwise.}\end{cases}
\end{align}

The value function $V(x', z)$ is infinite if the control problem of steering a point from $x'$ to $X_u$ is infeasible within the performance budget $J(w)\leq z$.
The value function of \eqref{eq:crash_traj_z} when restricted to the single initial condition $x'$ is
\begin{align}
\label{eq:value_q}
    Q(x') = \inf_{z \in [0, J_{\max}]} V(x', z).
\end{align}

The value function $Q(x')$ will have an upper bound of $J_{\max}$ if $Q(x')$ is finite, and otherwise will have a value of $\infty$. We make no assumptions of continuity or boundedness of $Q(x')$, beyond A1's assurance that $J_{\max}$ is finite.

\subsection{Subvalue Approximations}

We now use the moment-\ac{SOS} hierarchy to develop subvalue maps to lower-bound $Q(x')$ from \eqref{eq:value_q}.

\begin{prop}
\label{prop:v_subvalue}
Any function $v(t, x, z)$ that satisfies \eqref{eq:crash_cont_p} and \eqref{eq:crash_cont_lie} obeys $v(0, x, z) \leq V(x', z)$ from \eqref{eq:value_v} at all $(x, z) \in X \times Z$.
\end{prop}
\begin{proof}
Equations \eqref{eq:crash_cont_p} and \eqref{eq:crash_cont_lie} are inequality constrained versions of the \ac{HJB} equality constraints for an optimal value function $v^*$ \cite{liberzon2011calculus}:
\begin{subequations}
\begin{align}
    &v^*(t, x', z) = z       & & \forall (t, x', z) \in [0, T] \times X_u \times Z \\
    &\min_{w \mid (w, z) \in \Omega } \Lie_f v^*(t, x', z, w) = 0 & & \forall (t, x', z) \in [0, T] \times X \times Z.
\end{align}
\end{subequations}
Refer to the Section 4 of \cite{henrion2008nonlinear} and the proof of Proposition 1 of \cite{jones2021polynomial} for the establishment of subvalue relations.
\end{proof}

Let $\varphi \in \Mp{X}$ be a probability distribution with easily computable moments (e.g., uniform distribution over $X$ when $X$ is a ball or a box),  and $Q_{\max} \geq J_{\max}$ be a finite control cap. 
\begin{thm}
The following program provides a subvalue function $q(x) \leq Q(x)$:
\begin{subequations}
\label{eq:q_joint}
\begin{align}
    J^* &= \sup \int_X q(x) d\varphi(x) \label{eq:q_joint_obj}\\
    & q(x) \leq v(0, x, z) & & \forall (x, z) \in X \times [0, Z_{\max}] \label{eq:q_joint_qv} \\
    & q(x) \leq Q_{\max} & & \forall x \in \supp{\varphi} \label{eq:q_joint_cap} \\
    & z \geq v(t,x, z) & & \forall (t, x, z) \in [0, T] \times X_u \times Z  \label{eq:q_joint_z}\\
    & \Lie_{f} v(t, x, z, w) \geq 0 & & \forall (t, x, z, w) \in [0, T] \times X \times \Omega \label{eq:subvalue_lie_v} \\
    & v \in C^1([0, T] \times X \times Z) \label{eq:q_joint_v} \\
    & q \in C(X).
\end{align}
\end{subequations}
\end{thm}
\begin{proof}
Proposition \ref{prop:v_subvalue} proves that $v(0, x, z) \leq V(x, z)$ from \eqref{eq:value_v}. Constraint  \eqref{eq:q_joint_qv} imposes that $q(x) \leq v(0, x, z) \leq V(x, z)$ for all $x \in X$, which implies that $q(x) \leq \inf_z v(0, x, z)$ for all $x \in X$. From the definition of $Q(x')$ in \eqref{eq:value_q} with $Q(x') \leq \inf_z V(x', z)$, it therefore holds that $q(x) \leq Q(x)$ for all $x \in X$.
\end{proof}

\begin{cor}
The objective $J^*$ from \eqref{eq:q_joint} is finite and is bounded above by $J^* \leq Q_{\max}$.
\end{cor}
\begin{proof}
Constraint \eqref{eq:q_joint_cap} requires that $q(x)$ is upper-bounded by $Q_{\max}$. The objective \eqref{eq:q_joint_obj} is therefore upper-bounded by 
\begin{equation}
    \int_X q(x) d \varphi(x) \leq \int_X Q_{\max} d \varphi(x) \leq  Q_{\max} \int_{X} d \varphi(x) = Q_{\max},
\end{equation} given that $\varphi$ is a probability distribution.
\end{proof}

\begin{rmk}
\label{rmk:subvalue_exceed}
Let $v$ be a subvalue solution to \eqref{eq:q_joint_z}-\eqref{eq:q_joint_v}. Any point $x' \in X$ such that $\inf_{z \in Z} v(0, x', z) > J_{\max}$ implies that $Q(x') = \infty$.
\end{rmk}

\begin{rmk}
Without the $Q_{\max}$ cap on the value of $q$ in constraint \eqref{eq:q_joint_cap}, the optimal value of \eqref{eq:q_joint} could be $J^* = \infty$ if $\exists x' \in X: Q(x') = \infty$.
\end{rmk}

\begin{rmk}
    The Lie constraint in \eqref{eq:subvalue_lie_v} may be robustified through the methods in Section \ref{sec:crash_robust} when $f$ is input-affine and $W$ is a semidefinite representable set (e.g., a polytope from the $L_\infty$ data-driven case).
\end{rmk}

Define $q_d\in \R[x]_{\leq 2d}, v_d \in \R[t, x, z]_{\leq 2d}$ as the polynomials obtained by solving the degree-$d$ \ac{SOS} tightening of \eqref{eq:q_joint}. 
Let $I_u(x)$ be the indicator function
\begin{equation}
    I_u(x) \rw{=} \begin{cases}
    0 & x \in X_u \\
    -\infty & x \not \in X_u
    \end{cases}.
\end{equation}

For a sequence of orders $d' = 1..d$, a parametric function $q_{1:d}$ may be defined as
\begin{align}
\label{eq:parametric_subvalue}
    q_{1:d}(x) = \max(I_u(x), \max_{d' \in 1..d} q_{d'}(x)).
\end{align}

\begin{defn}[\cite{lasserre2010joint}] A sequence of continuous functions $\{q_k(x)\}$ converges \textbf{almost uniformly} to $Q(x)$ with respect to a measure $\varphi \in \Mp{X}$ if $\epsilon > 0: \exists  A \subseteq X$, such that $q_k \rightarrow Q$ uniformly on $X \setminus A$  and $\varphi(A) < \epsilon$.
\end{defn}

\begin{thm}
The function $q_{1:d}(x)$ will converge almost uniformly to \\ $\min(Q_{\max}, Q(x))$ on the state-space $\supp{\varphi} \in X$ as $d\rightarrow \infty$.
\end{thm}
\begin{proof}
Let $\tilde{v} \in \R[t, x, z]$ be a polynomial subvalue function that obeys \eqref{eq:q_joint_z}-\eqref{eq:q_joint_v}. 
Corollary 2.5 of \cite{lasserre2010joint} proves that the parameterized program $q_{1:d}$ will converge $\varphi$-almost uniformly to $\min(Q_{\max}, \min_{z} \tilde{v}(0, x, z))$, resulting in 
\begin{align}
    \lim_{k \rightarrow \infty} \int_{X} \abs{Q(x) - q_k(x)}d\varphi(x) = 0. \label{eq:almost_uniform}
\end{align} 
Increasing the degree of sublevel polynomials $\tilde{v}$ allows for the choice of admissible $\tilde{v}$ such that $\tilde{v}(0, x, z)$ converges in an $L^1$-sense to $V(x, z)$ whenever $V(x, z) \leq Q_{\max}$ \cite[Propositions 5 and 6]{jones2021polynomial}, thus proving the theorem.
\end{proof}



%
\section{Examples}
\label{sec:crash_examples}


This section demonstrates the utility of the crash-safety framework. Robust decompositions of the Lie constraint are applied in all examples. MATLAB R2021a code to generate examples is available at \url{https://github.com/Jarmill/crash-safety}. All \ac{SDP} are generated using YALMIP \cite{lofberg2004yalmip} and solved using Mosek \cite{mosek92}. Finite-degree crash-bounds from \eqref{eq:crash_sos_robust} are compared against \ac{OCP} bounds found using the solver CasADi \cite{andersson2019casadi}. 

The examples in \ref{sec:crash_subvalue_demo} perform crash safety with respect to applied inputs $w$ when the ground truth system is known. The examples in 
\ref{sec:crash_subvalue_data} perform data-driven crash-safety analysis.

\subsection{Single-Input Subvalue Comparison }
\label{sec:crash_subvalue_demo}
This example demonstrates the computation of crash-bounds and the creation of crash-subvalue functionals for system     \eqref{eq:flow_w1} with $J_{\max}=1$ and $Q_{\max} = 4$. This subvalue is constructed by solving \ac{SOS} tightenings of \eqref{eq:q_joint} in the space $X = [-2, 2]^2$ and in the time horizon $t \in [0, 5]$ 

\subsubsection{Half-Circle}
The first part of this example involves the half-circle  respect to the unsafe set $X_u = \{x \mid (x_1+0.25)^2 + (x_2+0.7)^2 \leq 0.5^2, \ (0.95+x_1+x_2)/\sqrt{2 }\leq 0\}.$
Figure \ref{fig:crash_subvalue} draws the unsafe set $X_u$ in red. The color shading (colorbar) plots $q_{1:5}(x)$ clamped to the range $[0, J_{\max}]=[0, 1]$. The integral objective values of \ac{SOS} tightening \eqref{eq:q_joint} at degrees $1..5$ are $J^*_{1:5} = [1.934\times 10^{-7}, 4.864\times 10^{-7}, 3.0794, 5.992, 8.260]$.

\begin{figure}[!h]
    \centering
    \includegraphics[width=0.7\linewidth]{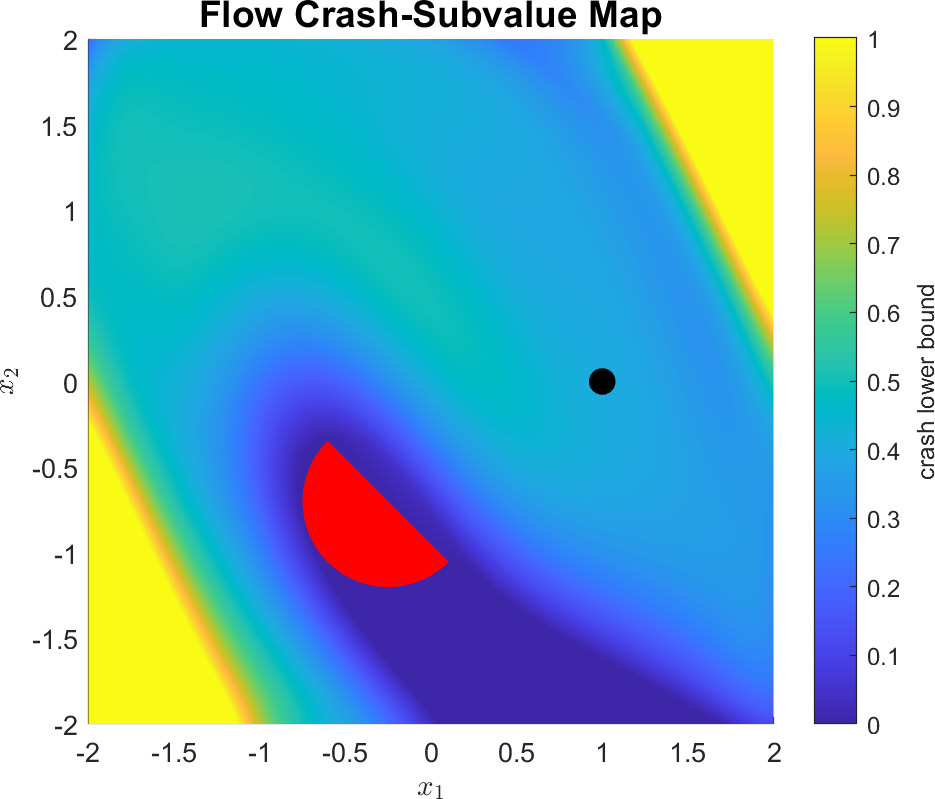}
    \caption{Subvalue function for Flow system \eqref{eq:flow_w1} between degrees $1..5$.}
    \label{fig:crash_subvalue}
\end{figure}
The black dot in Figure \ref{fig:crash_subvalue} is the specific initial point $X_0 = [1; 0]$. Table \ref{tab:crash_subvalue} lists crash-bounds on \eqref{eq:flow_w1} starting at $X_0$. The subvalue 
bound \eqref{eq:q_joint} is lower than the corresponding degree bounds at the $X_0$-specific program \eqref{eq:crash_cont}.

\begin{table}[h]
\centering
\caption{\label{tab:crash_subvalue} Crash-bounds at $X_0 = [1; 0]$ under \ac{SOS} tightenings}
\begin{tabular}{llllll}
order    & 1                     & 2      & 3      & 4   & 5  \\
subvalue \eqref{eq:q_joint}& $1.089\times10^{-9}$ & $1.607\times10^{-9}$  &  0.1473 & 0.3392 & 0.4053 \\
specific \eqref{eq:crash_cont} & $1.117\times10^{-7}$                & 0.1843 & 0.4369 & 0.5092 & 0.5118
\end{tabular}
\end{table}

We now consider worst-case crash-bounds for the half-circle set with respect to the perturbed flow system \eqref{eq:flow_w1} and the circular initial set $X_0 = \{x \mid 0.4^2 \geq (x_1-1)^2+x_2\}$. Crash-bounds as computed by \eqref{eq:crash_sos_robust} (\ac{SOS} tightenings to \eqref{eq:crash_cont}) in degrees $1..5$ are $[8.101\times 10^{-8}, 6.590\times 10^{-2}, 0.4054, 0.4631, 0.4638].$ The degree-5 lower-bound of $0.4638$ should be compared against the numerical bound of $0.4639$ produced by CasADi. The numerically solved trajectory (blue curve) is plotted in Figure \ref{fig:crash_circ_circ}, along with the unsafe set $X_u$ (red half-circle) and the initial set $X_0$ (black circle). The initial point of the controlled trajectory (blue dot) is $x_0 \approx [1.3424; 0.2069]$.

\begin{figure}[h]
    \centering
    \includegraphics[width=0.7\linewidth]{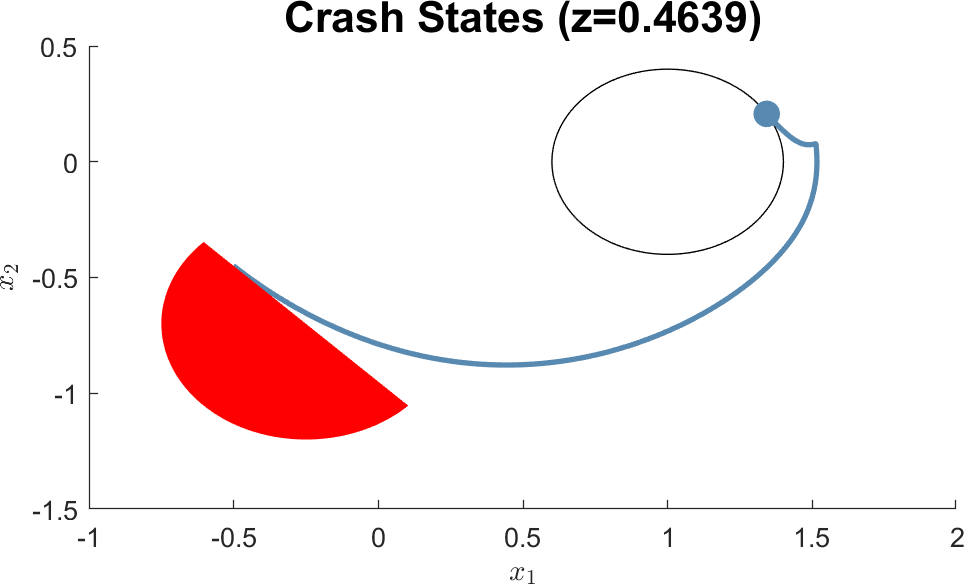}
    \caption{Numerical optimal control yields worst-case $Q^* \approx 0.4639$ for the half-circle $X_u$}
    \label{fig:crash_circ_circ}
\end{figure}

\subsubsection{Moon}
The second part of this example has a nonconvex moon-shaped unsafe set 
\begin{equation}
X_u = \{x \mid 0.8^2 - (x_1 - 0.4)^2 - (x_2 + 0.4)^2 \geq 0, \ (x_1 - 0.6596)^2 + (x_2 - 0.3989)^2-1.16^2 \geq 0\}. \label{eq:crash_moon}    
\end{equation}
Figure \ref{fig:casadi_moon} displays a controlled trajectory (blue curve)  starting from $X_0 = [0; 0]$ (black circle) and terminating in the $X_u$ (red moon), as computed by CasADi.

\begin{figure}[h]
    \centering
    \includegraphics[width=0.6\linewidth]{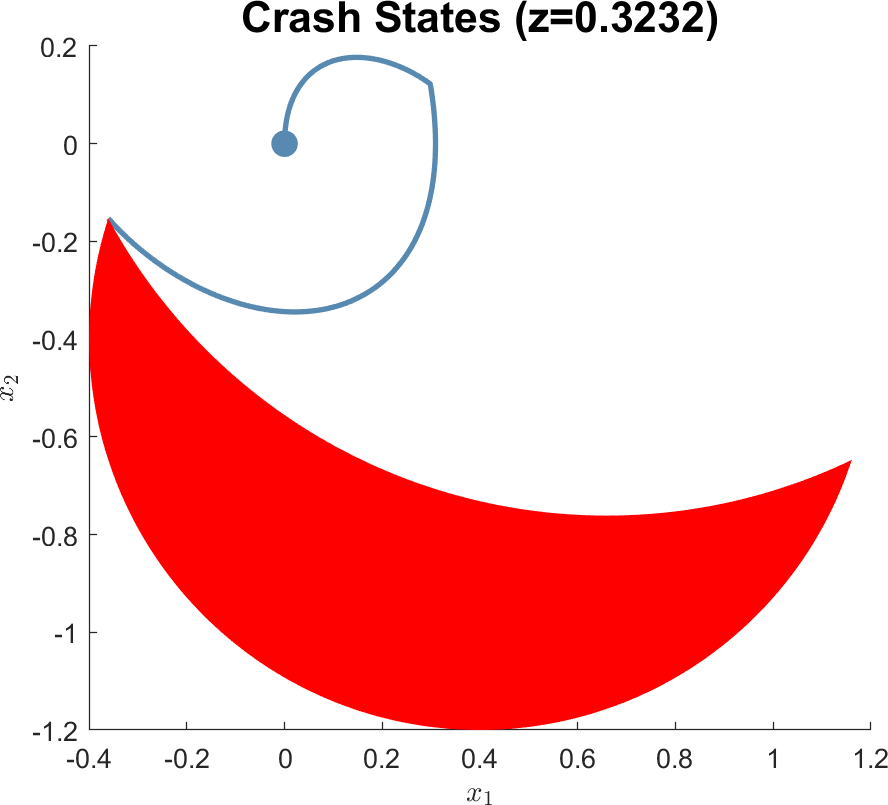}
    \caption{Numerical optimal control yields $Q^* \approx 0.3232$ for the moon $X_u$}
    \label{fig:casadi_moon}
\end{figure}

Table \ref{tab:crash_subvalue_moon} lists subvalue \eqref{eq:q_joint} and specific \eqref{eq:crash_cont} crash-bounds for $X_0 = [0; 0]$ between degrees $1..5$. The objectives of the \ac{SOS} tightenings to \eqref{eq:q_joint} are 
\[J^*_{1..5} = [1.973\times10^{-7}, 1.323\times10^{-7}, 1.027, 3.188, 4.502].\]

\begin{table}[h]
\centering
\caption{\label{tab:crash_subvalue_moon} Crash-bounds at $X_0 = [0; 0]$ for the moon \eqref{eq:crash_moon} under \ac{SOS} tightenings}
\begin{tabular}{llllll}
order                         & 2      & 3      & 4   & 5  \\
subvalue \eqref{eq:q_joint}   &$4.652\times10^{-10}$  &$-7.861\times10^{-2}$  &$-5.692\times10^{-3}$  &$7.721\times10^{-2}$ \\
specific \eqref{eq:crash_cont}     & 0.1010 & 0.2912 & 0.3216 & 0.3224
\end{tabular}
\end{table}
The data from Table \ref{tab:crash_subvalue_moon} at order 1 is subvalue: $8.770\times10^{-9}$, specific: $2.723\times10^{-8}$ (suppressed for layout/formatting purposes).

Figure \ref{fig:subvalue_moon} plots the subvalue function $q_{1:5}(x)$ from \eqref{eq:parametric_subvalue} under a cap of $Q_{\max} = 2$ (and $J_{\max} = 1$). All values of $q_{1:5}$ in Figure \ref{fig:subvalue_moon} are clamped to $[0, J_{\max}]$.

\begin{figure}[h]
    \centering
    \includegraphics[width=0.6\linewidth]{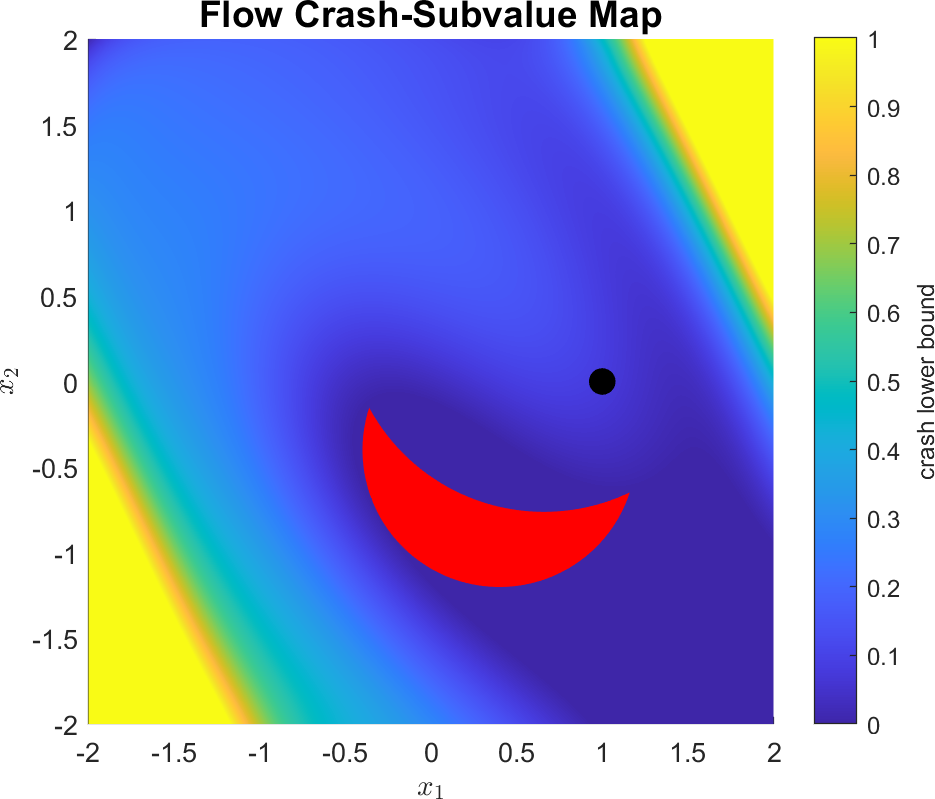}
    \caption{Subvalue map for the moon \eqref{eq:crash_moon} on the flow system \eqref{eq:flow_w1}}
    \label{fig:subvalue_moon}
\end{figure}

\subsection{Data-Driven Flow System}
\label{sec:crash_subvalue_data}

Data $\dc$ is collected for the Flow system \eqref{eq:flow} from $N_s=40$ samples with perfect knowledge in dynamics $\dot{x}_1 = x_2$ and a ground-truth \rw{uncertainty} bound of $\epsilon = 0.5$ in the coordinate $\dot{x}_2$. The noisy derivative data in $\dc$ and ground-truth derivatives are drawn in the orange and blue arrows respectively in Figure \ref{fig:flow_observations}.
It is assumed that $\dot{x}_2$ is described by a cubic polynomial in $(x_1, x_2)$. The parameterized polytope $\{w \mid A w \leq b + z\}$ ($\Omega$ with fixed $z$ value) has $L=10$ dimensions and $m=2 nT = 80$. The minimum possible corruption while obeying \eqref{eq:dynamics_affine_true} under the cubic \rw{uncertainty} model is $\inf_{(w, z) \in \Omega} z = 0.4617$. 

\begin{figure}[h]
    \centering
    \includegraphics[width=0.6\linewidth]{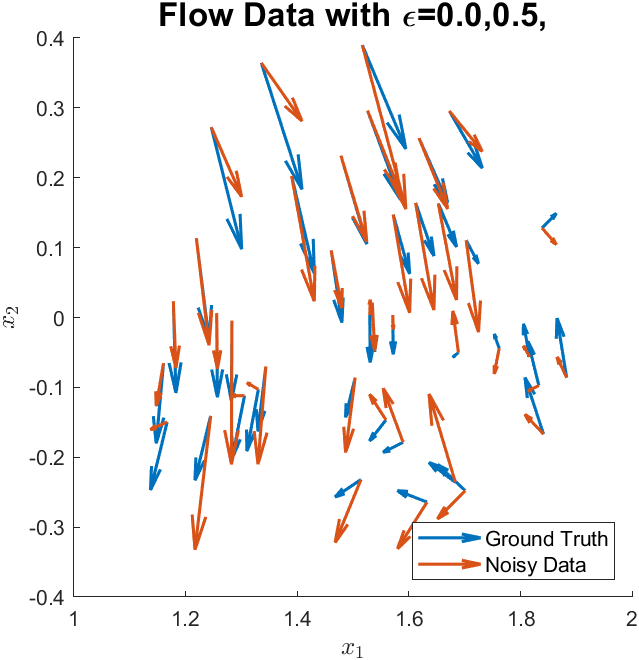}
    \caption{Observed data of the Flow system \eqref{eq:flow}}
    \label{fig:flow_observations}
\end{figure}

The crash-safety problem \eqref{eq:crash_cont} and subvalue problem \eqref{eq:q_joint} were solved with the unsafe set $X_u = \{x \mid (x_1+0.25)^2 + (x_2+0.7)^2 \leq 0.5^2, \ (0.95+x_1+x_2)/\sqrt{2 }\leq 0\}$ between $t=[0, 5]$ time units in the space $X = \{x \in \R^2 : \norm{x}_2^2 \leq 8\}$. The subvalue problem \eqref{eq:q_joint} integrates over the uniform measure of the ball $X$.

Table \ref{tab:crash_subvalue_data} reports bounds for the crash-corruption $Q(X_0)$ by solving Lie-robustified \ac{SOS} tightenings of \eqref{eq:crash_cont} and \eqref{eq:q_joint} from degrees $1..4$ with $J_{\max} = 1, \ Q_{\max}=4$. The objective function (integrals of $q(x)$) for the subvalue  \eqref{eq:q_joint} are $J^*_{1:4} = [0.2193, 3.8185, $ $7.8326, 18.5945]$. The subvalue-estimated control cost at $X_0$ between degrees $1..4$ is $0.3399$ by Equation \eqref{eq:parametric_subvalue}. The subvalue-estimated bound is valid for all $x \in X$, and is therefore lower than the bound $q^*_4 = 0.5499$ from \eqref{eq:crash_cont} that focuses exclusively on the initial point $X_0$.

\begin{table}[h]
\centering
\caption{\label{tab:crash_subvalue_data} Data-Driven Crash-bounds at $X_0 = [1; 0]$ under \ac{SOS} tightenings}
\begin{tabular}{lllll}
order    & 1                     & 2      & 3      & 4   \\
specific \eqref{eq:crash_cont} & 
$0.0582$ & $0.4423$  & 0.4864& 0.5499   \\
subvalue \eqref{eq:q_joint} & $6.180\times 10^{-3}$ & 0.1829 & 0.3399 & \rw{0.3399}\\
\end{tabular}
\end{table}

Figure \ref{fig:flow_data_driven_subvalue} plots the subvalue function from \eqref{eq:parametric_subvalue} on the data-driven flow system. Subvalues in the plot are clamped to the range $[0, J_{\max}] = [0, 1]$.
\begin{figure}[!h]
    \centering
    \includegraphics[width=0.6\linewidth]{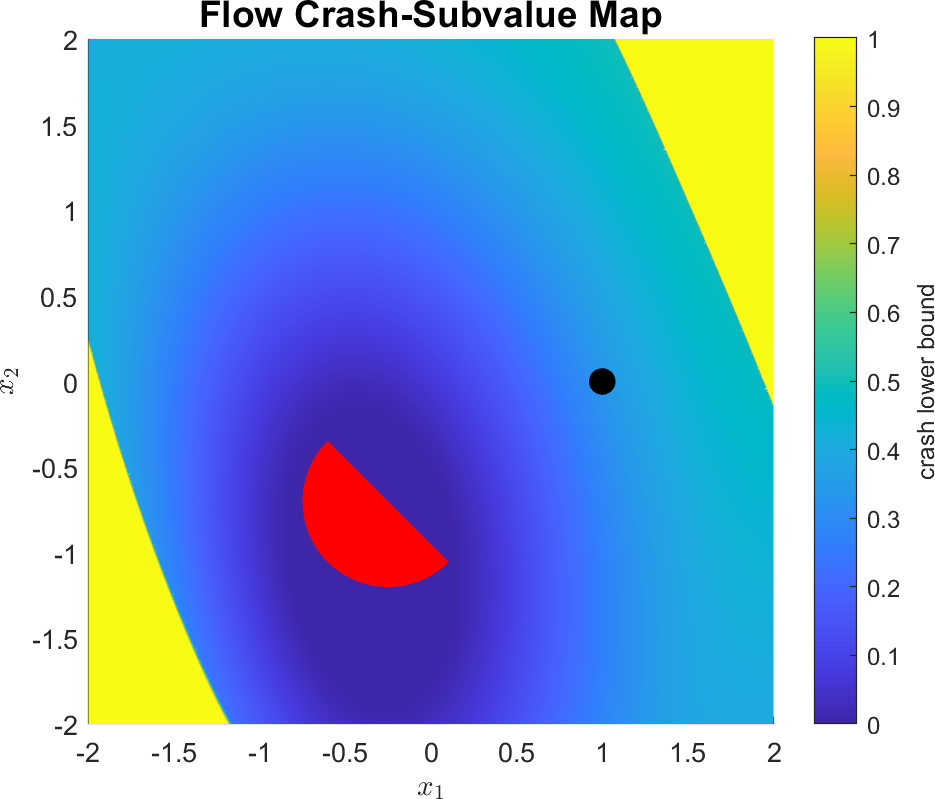}
    \caption{Subvalue for data-driven \eqref{eq:flow} between degrees $1..4$}
    \label{fig:flow_data_driven_subvalue}
\end{figure}

Safety of trajectories starting in $X_0$ is certified because the crash-bound $\tilde{q}^*_4 = 0.5499$ is greater than the ground-truth \rw{uncertainty}-bound $\epsilon = 0.5$.

\rw{The} CasADi optimal control suite \cite{andersson2019casadi} \rw{was used} to numerically solve the crash program \eqref{eq:crash_traj}, \rw{and the produced trajectory is visualized in Figure \ref{fig:flow_casadi}}. The numerical crash-bound of $q^{\textrm{CasADI}} = 0.5499$ is approximately equal (up to four decimal places) to the crash-bound $q^*_4=0.5499$.

\begin{figure}[h]
    \centering
    \includegraphics[width=0.7\linewidth]{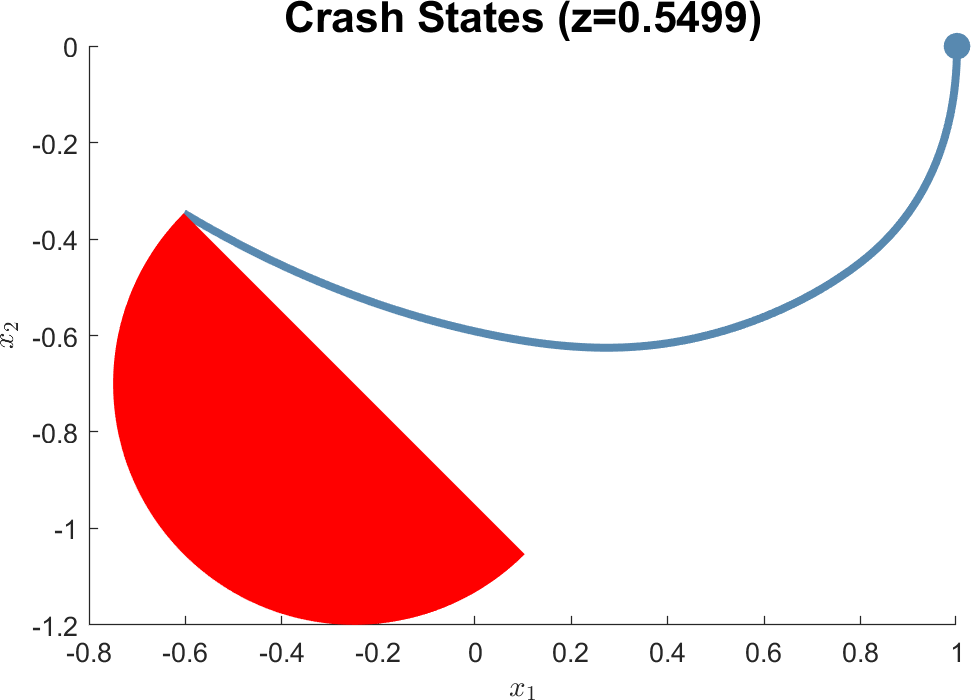}
    \caption{Numerically computed crash-bound for data-driven Flow \eqref{eq:flow}}
    \label{fig:flow_casadi}
\end{figure}

The left plot of figure \ref{fig:crash_casadi_control} shows the applied control of the $L=10$ inputs. The right plot demonstrates how the polytopic input constraint is obeyed with respect to the crash bound $q^{\textrm{CasADI}} = 0.5499$ (upper and lower black lines).

 \begin{figure}[!ht]
     \centering
     \begin{subfigure}[b]{0.48\linewidth}
         \centering
         \includegraphics[width=\linewidth]{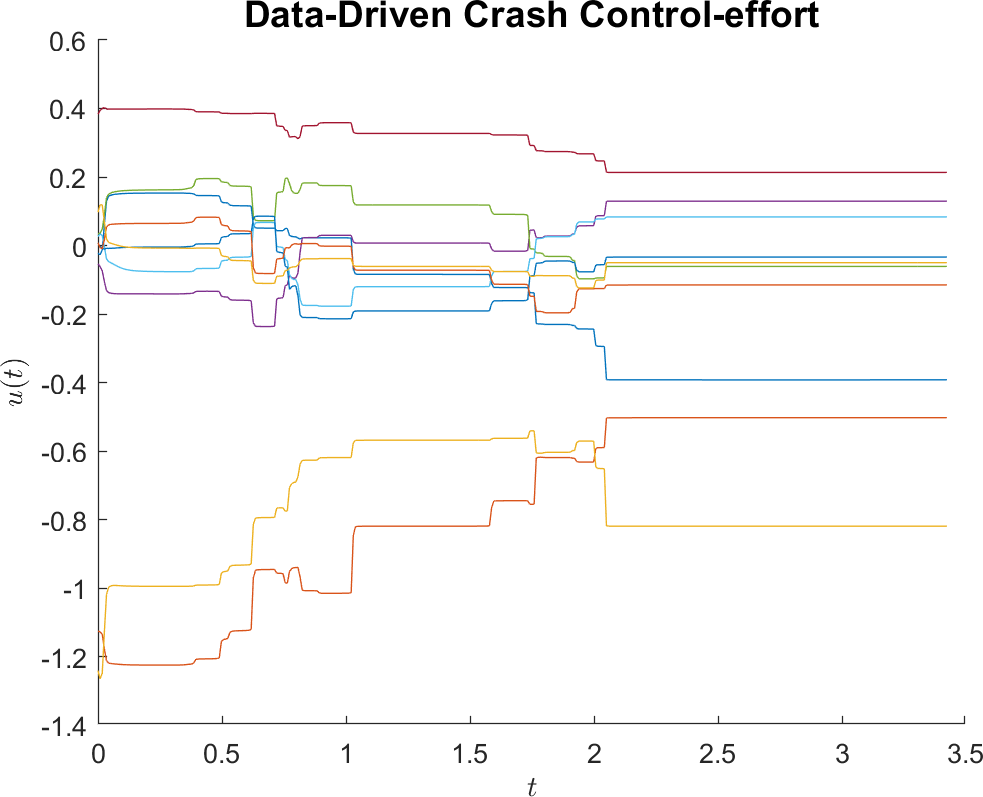}
     \end{subfigure}
     \;
     \begin{subfigure}[b]{0.48\linewidth}
         \centering
         \includegraphics[width=\linewidth]{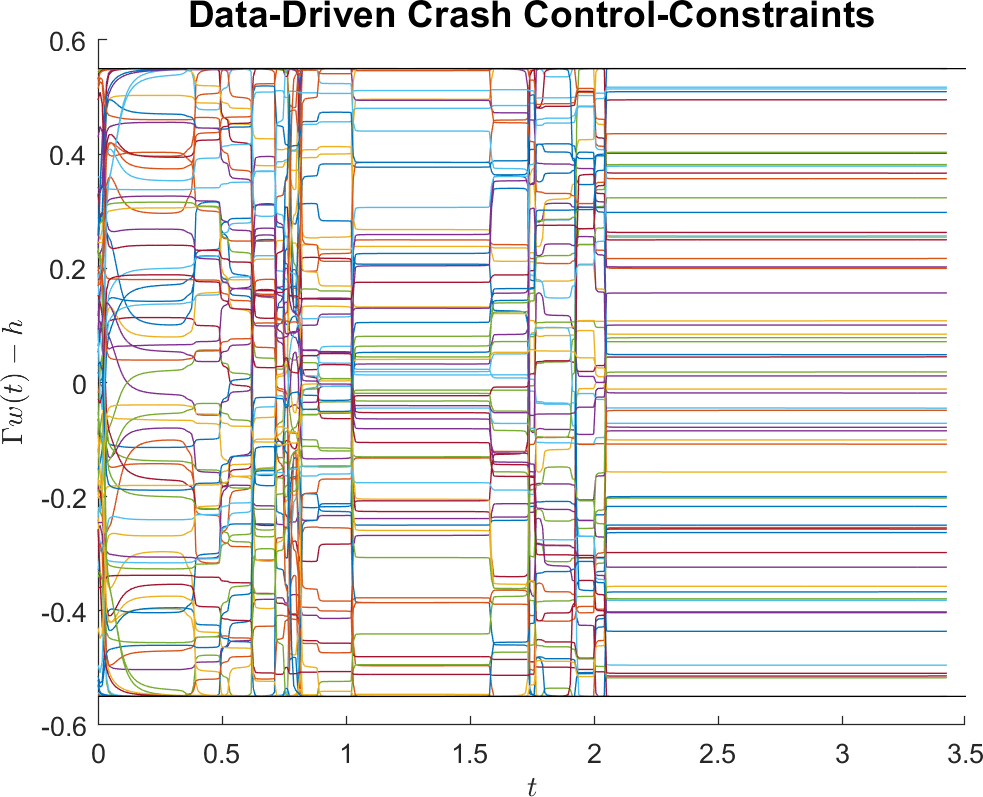}
     \end{subfigure}
      \caption{\label{fig:crash_casadi_control}Applied control for the data-driven Flow crash system.}
\end{figure}

 These crash-bounds should be compared against the $L_2$ distance estimates of $c^*_{1:5} = [1.698\times 10^{-5}, \ 0.1936, \ $ $0.2003, \ 0.2009, \ 0.2013]$  from Section 6.3 of \cite{miller2023robustcounterpart}. The distance estimates do not indicate that adding an additional budget of $0.0499$ constraint violation will cause at least one trajectory to enter the unsafe set.




\section{Conclusion}

\label{sec:conclusion}

This paper utilized peak minimizing control in order to perform safety analysis. The returned values from \ac{SOS} programs are lower-bounds on the maximal control effort needed to crash into the unsafe set.  Crash-safety adds a new perspective on the safety of trajectories, covering some of the blind spots of distance estimation and safety margins. Crash-safety may be applied in the context of data-driven systems analysis, by quantifying the minimum tolerable corruption in an \rw{uncertainty} model before a trajectory is at risk of being unsafe.

Future work involves attempting to reduce computational burden of the Crash programs \eqref{eq:crash_sos} by identifying new kinds of structure (in addition to robust decompositions) to hopefully allow for real-time computation. 
Other extensions could include applying these methods to other classes of systems (e.g., discrete-time, hybrid), and creating a stochastic interpretation of crash-safety.

\section*{Acknowledgements}

The authors thank Necmiye Ozay and Alain Rapaport for discussions regarding crash-safety and peak-minimizing control.




\bibliographystyle{IEEEtran}
\bibliography{references.bib}

\end{document}